\documentclass{article}
\usepackage{amssymb,amscd,amsthm,amsmath,color,mathdots,hyperref}

\newcommand{\PP}{\mathbb{P}}

\newcommand{\FFF}{\mathbb{F}}

\newcommand{\FF}{\mathcal{F}}

\newcommand{\UU}{\mathcal{U}}
\newcommand{\bP}{\mathbb{P}}
\newcommand{\gG}{\mathbb{G}}
\newcommand{\QQl}{\mathbb{Q}_\ell}

\DeclareMathOperator{\id}{id}

\newcommand{\codim}{\operatorname{codim}}

\newcommand{\image}{\operatorname{image}}
\newcommand{\CH}{\operatorname{CH}}
\newcommand{\cl}{\operatorname{cl}}

\newtheorem{theorem}{Theorem}[section]
\newtheorem{lemma}[theorem]{Lemma}
\newtheorem{proposition}[theorem]{Proposition}
\newtheorem{corollary}[theorem]{Corollary}
\newtheorem{definition}[theorem]{Definition}

\title{Rational curves on complete intersections in positive characteristic}
\author{Eric Riedl and Matthew Woolf}

\begin{document}
\maketitle

\begin{abstract}
We study properties of rational curves on complete intersections in positive characteristic. It has long been known that in characteristic 0, smooth Calabi-Yau and general type varieties are not uniruled. In positive characteristic, however, there are well-known counterexamples to this statement. We will show that nevertheless, a \emph{general} Calabi-Yau or general type complete intersection in projective space is not uniruled. We will also show that the space of complete intersections of degree $(d_1, \cdots, d_k)$ containing a rational curve has codimension at least $\sum_{i=1}^k d_i - 2n + 2$ in the moduli space of all complete intersections of given multidegree and dimension.
\end{abstract}

\section{Introduction}

In characteristic zero, there is to some extent a dichotomy in the behavior of rational curves on smooth complete intersections in projective space. If the complete intersection is Fano, then it is rationally connected, i.e., there is a rational curve connecting any two points. On the other hand, if the complete intersection is Calabi-Yau or of general type, then it is not even uniruled, i.e., there is no rational curve passing through a very general point.

In positive characteristic, the notions of rational connectedness and uniruledness become more complicated. While there are still notions of rational connectedness or uniruledness as above, we can alternatively require a variety to be \emph{separably} rationally connected or uniruled, which essentially means that there are rational curves on the variety which have many infinitesimal deformations.

It is still true that all Fano varieties are rationally connected, whereas Calabi-Yau and general type varieties are never separably uniruled. More recently, it has been shown that the general Fano complete intersection is even separably rationally connected \cite{CZ}. However, in positive characteristic there are general type varieties that are uniruled. Shioda constructed examples of smooth hypersurfaces of arbitrarily large degree which are unirational, so in particular, rationally connected and uniruled \cite{S}.  Liedtke shows that supersingular K3 surfaces are unirational \cite{L}, and also gives a construction of some families of uniruled surfaces of general type \cite{L2}.

In this paper, we will show that despite the existence of these pathological examples, we still have the following result.

\begin{theorem}
\label{nonuniruled}
Let $X$ be a general Calabi-Yau or general type complete intersection in projective space. Then $X$ is not uniruled.
\end{theorem}

Christian Liedtke has kindly pointed out to us that one can use isocrystal methods to prove this result for surfaces.

Using Theorem \ref{nonuniruled} together with the methods of \cite{RiedlYang}, we can also obtain more quantitative results about uniruled hypersurfaces and hypersurfaces containing rational curves. For instance, we show the following.

\begin{theorem}
For $d \geq 2n-1$, a very general hypersurface will contain no rational curves, and moreover, the locus of hypersurfaces that contain rational curves will have codimension at least $d-2n+2$.
\end{theorem}

This transports results of Ein to postitive characteristic \cite{ein}.

In the sequel, all Chow groups will have rational coefficients. We will work over a field $k$ of positive characteristic $p$ and fix a prime $\ell \neq p$.

We would like to thank Arend Bayer, Izzet Coskun, Johan de Jong, Lawrence Ein, H\'el\`ene Esnault, J\'anos Koll\'ar, Christian Liedtke, and Jason Starr for numerous helpful conversations. Both authors were partially supported by an NSF RTG grant during this research.

\section{Preliminaries on notions of uniruledness and rational connectedness}

There are a number of inequivalent notions of rational connectedness and uniruledness in positive characteristic, so in this section we will give the precise definitions we will use and some basic facts about them. Details can be found in \cite{K}.

\begin{definition}
A variety $X$ is \emph{rationally chain connected} if there are varieties $U$ and $Y$, a proper flat morphism $g:U \to Y$ of relative dimension one whose geometric fibers have all components rational, and a map $u:U \to X$ such that the natural map \[ u \times u: U \times_Y U \to X \times X\] is dominant.

$X$ is \emph{rationally connected} if we can choose $U$, $Y$, $g$, and $u$ as above with the additional property that the geometric fibers of $g$ are irreducible.

$X$ is \emph{separably rationally connected} if there is a variety $Y$ and a morphism $u:Y \times \bP^1 \to X$ such that the natural map \[ Y \times \bP^1 \times \bP^1 \to X \times X \] is dominant and separable.
\end{definition}

\begin{definition}
A variety $X$ is \emph{uniruled} if there is a variety $Y$ and a dominant map $u:Y \times \bP^1 \to X$.

It is separably uniruled if we can find a separable such $u$.
\end{definition}

If the base field is algebraically closed, these properties and their negations are preserved under extension to another algebraically closed ground field. We will say that a variety is geometrically uniruled if its base change to the algebraic closure of the ground field is uniruled.

If the ground field is uncountable, there are alternate characterizations of some of these near-rationality properties which can be easier to use.  For instance, a variety is rationally connected if and only if two very general points can be connected by an irreducible rational curve. It is rationally chain connected if two very general points can be connected by a chain of rational curves. It is uniruled if a rational curve can be found passing through a very general point.

\section{Non-uniruledness of general complete intersections}
In this section, we prove that a general complete intersection with effective canonical bundle is not uniruled.  The proof comes from an analysis of the coniveau filtration of the middle cohomology of $X$.  On the one hand we have the result of Katz that the coniveau type must be $0$.  On the other hand, we show that for a uniruled variety of dimension $n$, the coniveau of $H^n$ is strictly positive.  This allows us to conclude that a general such complete intersection is not uniruled.

\subsection{Coniveau filtration}
We begin by recalling the definition of the coniveau filtration on cohomology, and some of its properties, which is given in \cite{SGA 7 II}, expos\'e XX.



\begin{definition}
\label{corrConiveau}
\begin{equation*} N^j H^i(X_{{\overline{k}}},\mathbb{Q}_\ell)=\sum _{(d,T,Z)}\image \cl(Z)_*:H^{i-2j}(T_{{\overline{k}}},\mathbb{Q}_\ell) \to H^i(X_{{\overline{k}}},\mathbb{Q}_\ell) \end{equation*}
where $d$ runs over all nonnegative integers, $T$ runs over all smooth proper connected varieties of dimension $d$, $Z$ runs over all algebraic cycles in $T \times X$ of codimension $j+d$, and $\cl(Z)_*$ is the map on cohomology induced by the correspondence $Z$, i.e., \[\cl(Z)_*(\alpha)=\pi_{2*}([Z] \cup \pi_1^*(\alpha)).\]
\end{definition}

We will say that the \emph{coniveau type} of $H^i(X_{{\overline{k}}},\QQl)$ is the greatest integer $j$ such that $H^i(X_{{\overline{k}}},\QQl)=N^j H^i(X_{{\overline{k}}},\QQl)$

\begin{lemma}
\label{functorialConiveau}
Let $Z \in A^{\dim X+k}(X \times Y)$. The induced map \[ Z_*:H^i(X_{\overline{k}},\QQl) \to H^{i+2k}(Y_{\overline{k}},\QQl) \] sends $N^j H^i(X_{\overline{k}},\QQl)$ to $N^{j+k}H^{i+2k}(Y_{\overline{k}},\QQl)$.
\end{lemma}
\begin{proof}
We know that $N^j H^i(X_{\overline{k}},\QQl)$ is spanned by classes $\alpha$ for which there is a smooth projective variety $T$ of dimension $d$, a cycle $\Gamma \in A^{d+j}(T \times X),$ and a class $\beta \in H^{i-2j}(T_{\overline{k}},\QQl)$ such that $\Gamma_* \beta=\alpha$. It therefore suffices to prove this result for such classes.

But we know that $Z_* \alpha=(Z \circ \Gamma)_* \beta$, and $Z \circ \Gamma \in A^{d+k+j}(T \times Y)$, so it follows from the definition of the coniveau filtration that \[Z_* \alpha \in N^{j+k}H^{i+2k}(Y_{\overline{k}},\QQl)\].
\end{proof}

Given a smooth complete intersection of dimension $n$ in projective space, the Lefschetz hyperplane theorem determines that its cohomology is the same as projective space in all degrees but $H^n$. The following key result gives us information about the coniveau of this interesting cohomology group.

\begin{theorem}[Katz, \cite{SGA 7 II}, expos\'e XXI]\label{complete-intersection-coniveau}
Let $X$ be a general Calabi-Yau or general type $n$-dimensional complete intersection in projective space. Then the $n$th cohomology of $X$ has coniveau type 0.
\end{theorem}

We also have a result which will allow us to give explicit examples of such complete intersections.

\begin{proposition}[\cite{SGA 7 II}, expos\'es XIX, XXI, et XXII] \label{coniveau-congruence} 
Let $k=\FFF_q$ be a finite field. If all the higher $\ell$-adic cohomology groups of $X$ have positive coniveau type, then \[ |X(k)| \equiv 1 \pmod q \]
\end{proposition}

\subsection{Coniveau of uniruled varieties}

This subsection will be devoted to showing that a uniruled variety has positive coniveau type in its middle cohomology.  The proof, which is inspired by Esnault \cite{E}, has three steps.  First, we show that $\CH_0$ of uniruled varieties comes from $\CH_0$ of a proper subvariety.  Second, we recall a result of Bloch and Srinivas that shows that such varieties admit a decomposition of the diagonal of a certain form.  Finally, we use this decomposition of the diagonal and follow Bloch-Srinivas and Voisin to show that the middle cohomology of varieties with such a decomposition of the diagonal must have positive coniveau type.

\begin{proposition}
Let $X/k$ be a smooth projective variety with $k$ uncountable and algebraically closed. Suppose that $X$ is uniruled. Then there is a proper closed subvariety $Z \subset X$ such that the natural map $\CH_0(Z) \to \CH_0(X)$ is surjective, or equivalently, there is a nonempty open subset $V \subset X$ such that $\CH_0(V)=0$.
\end{proposition}
\begin{proof}
Let $Z$ be an ample divisor on $X$. Let $p \in X$ be arbitrary. We show that the $0$-cylce $p$ is equivalent to a cycle supported on $Z$, which will suffice to prove the claim. Since $X$ is uniruled, there is a rational curve $C$ in $X$ through $p$. Since $Z$ is ample, $C$ will meet $Z$. Moving $p$ along $C$, we see that $p$ is equivalent to a cycle supported on $Z$.



\end{proof}

In fact, as Jason Starr pointed out to us, one can obtain a stronger result. Namely, using Nadel's trick and the MRCC fibration as described by Debarre \cite{De}, one can prove that any variety of Picard rank 1 that is uniruled must also be rationally connected. By working in families, it is possible to prove this even for complete intersection surfaces, whose Picard numbers are not necessarily 1.

We now recall the decomposition of the diagonal, proposition 1 of \cite{BS}.

\begin{proposition}\label{decompdiagonal}
Let $X$ be a smooth projective variety of dimension $n$ and $Z \subset X$ a closed subvariety such that $\CH_0(Z) \to \CH_0(X)$ is surjective. Then there is a positive integer $N$, a divisor $D$, and cycles $\Gamma_1$, $\Gamma_2 \in \CH_n(X \times X)$ such that $\Gamma_1$ is supported on $Z \times X$, $\Gamma_2$ is supported on $X \times D$ and \[ N[\Delta]=\Gamma_1+\Gamma_2.\]
\end{proposition}

Having a decomposition of the diagonal has strong implications on the coniveau type of $H^i(X_{\overline{k}},\QQl)$. The following proposition is essentially theorem 3.16 of \cite{V}, and the proof is very similar, with some minor adjustments needed for it to work in positive characteristic.

\begin{proposition} \label{coniveauchow}
Let $X$ be a smooth projective variety of dimension $n$. Let $Z \subset X$ be a closed subvariety of dimension $i$ such that $\CH_0(Z) \to \CH_0(X)$ is surjective. Then for $m>i$, $H^m(X_{{{\overline{k}}}},\QQl)$ has strictly positive coniveau type.
\end{proposition}

Before we give the proof, we will also need the following lemma.

\begin{lemma}
\label{hardLefschetz}
Let $Z$ be a smooth projective variety of dimension $n$. Then if $m=n+i$, $H^m(Z_{\overline{k}},\QQl)$ has coniveau type at least $i$.
\end{lemma}
\begin{proof}
Let $L$ be the class of an ample divisor on $Z$. By the hard Lefschetz theorem for $\ell$-adic cohomology (see \cite{D} or \cite{M}), $\cdot \cup L^i:H^{n-i}(Z_{{\overline{k}}},\QQl) \to H^{n+i}(Z_{{\overline{k}}},\QQl)$ is an isomorphism. But this map is induced by a correspondence, namely $\Delta_{*}(L^i)$, so its image lies in $N^iH^{n+i}(Z_{\overline{k}},\QQl)$ by Definition \ref{corrConiveau}.
\end{proof}

We can now give the proof of proposition \ref{coniveauchow}.

\begin{proof}
We use the decomposition of the diagonal $N[\Delta]=\Gamma_1+\Gamma_2$ of proposition \ref{decompdiagonal}, with $\Gamma_1$ supported on $X \times D$ and $\Gamma_2$ supported on $Z \times X$. Note that the diagonal correspondence induces the identity map on cohomology. It therefore suffices to show that the images of $(\Gamma_1)_*$ and $(\Gamma_2)_*$ lie inside the coniveau 1 part of $H^m(X_{\overline{k}},\QQl)$.

We begin with $(\Gamma_{2})_*$. Recall that an alteration of a variety $V$ is a generically finite, proper map $\tilde{V} \to V$ with $\tilde{V}$ smooth.  Any variety over any field admits an alteration \cite{DJ}. Take $\tilde{D}$ an alteration of $D$, with $f:\tilde{D} \to X$ the natural map. Then $(\Gamma_{2})_*$ factors through $f_*$, since $\Gamma_{2}$ is a pushforward of a rational cycle $\tilde{\Gamma}_2$ on $X \times \tilde{D}$, and $(\Gamma_{2})_* = f_* \circ (\tilde{\Gamma}_{2})_*$.  The image of $f_*$ is contained in $N^1 H^m(X_{\overline{k}},\QQl)$ by Definition \ref{corrConiveau} because $f$ is generically finite onto a proper subvariety of $X$ and pushforward by such a map increases cohomological degree.

We now show the result for $(\Gamma_1)_*$. Let $\tilde{Z}$ be an alteration of $Z$ and $g:\tilde{Z} \to X$ the natural map. As before, we know there is a cycle $\tilde{\Gamma}_1$ on $\tilde{Z} \times X$ which pushes forward to $\Gamma_1$ under $g \times \id$. We have $(\Gamma_{1})_*=(\tilde{\Gamma}_{1})_* \circ g^*$. It suffices to show that the image of $g^*$ is contained in $N^1H^{m}(\tilde{Z}_{\overline{k}}, \QQl)$ since the coniveau filtration is preserved under pushforward by lemma \ref{functorialConiveau}. By hypothesis, $m$ is greater than the dimension of $\tilde{Z}$, so by Lemma \ref{hardLefschetz}, it has positive coniveau type.
\end{proof}

By Theorem \ref{complete-intersection-coniveau}, we know that the middle cohomology of a general non-Fano complete intersection has coniveau type 0, but this is impossible for uniruled varieties, so we can deduce the following theorem.

\begin{theorem}\label{notUniruledProp}
Let $X$ be a smooth projective variety such that its middle cohomology has coniveau type 0. Then $X$ is not uniruled. In particular, a general non-Fano complete intersection is not uniruled.
\end{theorem}

Using Proposition \ref{coniveau-congruence} to relate coniveau to point-counting, we deduce the following corollary.

\begin{corollary} \label{pointCountUniruled}
Let $X$ be a smooth complete intersection in projective space over a finite field $k=\FF_q$. If \[|X(k)| \not\equiv 1 \pmod q \] then $X$ is not geometrically uniruled.
\end{corollary}

\subsection{Fermat hypersurfaces}

Some of the earliest examples of unirational hypersurfaces have been Fermat hypersurfaces. Suppose always that \[ d \not \equiv 0 \pmod p. \] Shioda and Katsura showed in \cite{SK} that for odd $n \geq 3$ and $d \geq 4$, the Fermat hypersurface of degree $d$ in $\bP^n$ is unirational if there is an integer $\nu$ such that \[p^\nu \equiv -1 \pmod d.\] If $n=3$, then the converse also holds.

By using Corollary \ref{pointCountUniruled}, we can give examples of non-uniruled Fermat hypersurfaces. Given a hypersurface over a finite field, we can use the Fulton trace formula \cite{F} to count the number of points modulo $p$. More precisely, we have the following corollary of the trace formula, which is proposition 5.15 of \cite{M2}.

\begin{proposition}
If $X=V(F)$ is a hypersurface in $\PP^n$ of degree $n+1$, then \[|X(\FFF_Q)| \not\equiv 1 \pmod p \] if and only if the coefficient of $(x_0\cdots x_n)^{p-1}$ in $F^{p-1}$ is nonzero.
\end{proposition}

A straightforward calculation then implies that if $d=n+1$ and \[ p \equiv 1 \pmod d \] then this coefficient is nonzero, so in particular, these Fermat hypersurfaces are not geometrically uniruled.

\section{Non-existence of rational curves}
In characteristic $0$, Ein \cite{ein} proves that a very general hypersurface in $\PP^n$ of degree $d$ with $d \geq 2n-1$ contains no rational curves.  Voisin \cite{voisin} slightly strengthens this, proving that a very general hypersurface of degree $d \geq 2n-2$ contains no rational curves.  Their techniques become more complicated in characteristic $p$.  In particular, they are dealing with positivity of certain vector bundles under various finite covers, and in characteristic $p$, it is not immediately clear how to show that these behave well under inseparable base change. We present an alternate proof of Ein's result using strictly characteristic $p$ techniques which work for all complete intersections. We do not know of a proof of Voisin's result in positive characteristic.


Consider the universal complete intersection of multidegree $\overline{d}=(d_1,\ldots,d_i)$ in $\PP^n$, $\UU_{n,\overline{d}} = \{(p,X) | p \in X,\,X \in \mathcal{H}_{\overline{d}}\}$, where $\mathcal{H}_{\overline{d}}$ is the Hilbert scheme of complete intersections in $\PP^n$ of multidegree $\overline{d}$, and let $R_{n,\overline{d}} \subset \UU_{n,\overline{d}}$ be the space of pairs $(p,X)$ with a rational curve in $X$ passing through $p$.  $\UU_{n,\overline{d}}$ is a projective variety, and $R_{n,\overline{d}}$ will be a countable union of projective varieties.  Our main result is the following.

\begin{theorem}
\label{hyperbolicityThm}
If $R_{n,\overline{d}}$ is a proper subset of $\UU_{n,\overline{d}}$, then for any $c \geq 0$, every component of $R_{n-c,\overline{d}}$ is codimension at least $c+1$ in $\UU_{n,\overline{d}}$.
\end{theorem}

Combined with Proposition \ref{notUniruledProp}, this gives

\begin{corollary}
The space of multidegree $\overline{d}$ complete intersections in $\PP^n$ containing a rational curve has codimension at least $\sum d_i-2n+2$. The space of uniruled hypersurfaces has codimension at least $\sum d_i-n$.
\end{corollary}
\begin{proof}
By Proposition \ref{notUniruledProp}, we see that $R_{\sum d_i-1,\overline{d}}$ is codimension at least $1$ in $\UU_{\sum d_i-1,\overline{d}}$, so it follows from Theorem \ref{hyperbolicityThm} that for $n<\sum d_i-1$, $R_{n,\overline{d}}$ is codimension at least $\sum d_i-n$ in $\UU_{n,\overline{d}}$. The fibers of the map from $R_{\sum d_i-c,\overline{d}}$ to the space of all hypersurfaces containing a rational curve are at least one-dimensional, so we can conclude that the codimension of the space of hypersurfaces containing a rational curve in the space of all hypersurfaces of degree $d$ in $\PP^n$ is at least $\sum d_i-2n+2$. The fibers of the map from $R_{n,\overline{d}}$ to the space of hypersurfaces have dimension $n-1$ when considering uniruled hypersurfaces, so we conclude that the space of uniruled hypersurfaces has codimension at least $\sum d_i-n$.
\end{proof}

The techniques of the proof of Theorem \ref{hyperbolicityThm} are similar to those in \cite{RiedlYang}, but we reproduce them here because the statement that we are proving is slightly different and because we want to emphasize the independence of characteristic.  We first need an elementary result on Grassmannians that we quote from \cite{RiedlYang}.

\begin{proposition}
\label{GrassProp}
Let $m \leq n$.  Let $B \subset \gG(m,n)$ be irreducible of codimension at least $\epsilon \geq 1$.  Let $C \subset \gG(m-1,n)$ be a nonempty subvariety satisfying the following condition: $\forall c \in C$, if $b \in \gG(m,n)$ has $c \subset b$, then $b \in B$.  Then it follows that the codimension of $C$ in $\gG(m-1,n)$ is at least $\epsilon + 1$.
\end{proposition}

\begin{proof} (Theorem \ref{hyperbolicityThm}) 
Let $(p,X) \in R_{n-c,\overline{d}}$ be a general point on one of the components.  We wish to show that the codimension of any component of $R_{n-c,\overline{d}}$ passing through $(p,X)$ is at least $c$.  In order to show this, we will find a subvariety $\FF \subset \UU_{n,\overline{d}}$ such that $\codim(\FF \cap R_{n-c,\overline{d}} \subset \FF) \geq c$.

Let $(p,Y) \in \UU_{n,\overline{d}}$ be very general, so that there are no rational curves in $Y$ through $p$.  Let $(p,Z) \in \UU_{M,\overline{d}}$ be a pair such that $(p,Y)$ and $(p,X)$ are both parameterized linear sections of $(p,Z)$, where $M$ is some large number .  Let $\FF_m$ be the closure of the space of parameterized linear sections of $(p,Z)$ in $\UU_{m,d}$.  By hypothesis, 
\[ \codim (\FF_n \cap R_{n,\overline{d}} \subset \FF_n) \geq 1, \] 
so it follows by Proposition \ref{GrassProp} that 
\[ \codim (\FF_{n-c} \cap R_{n-c,\overline{d}} \subset \FF_n) \geq c+1 \]
which concludes the proof.
\end{proof}

\bibliographystyle{plain}

\end{document}